\documentclass{amsart}

\usepackage[latin1]{inputenc}
\usepackage[T1]{fontenc}
\usepackage{amsmath,amsfonts,amssymb,stmaryrd,amsthm,url,geometry,mathrsfs,esint,bm,tikz}
\usepackage{mathtools}
\usepackage{braket,verbatim}
\usepackage{breqn,verbatim}
\usepackage{xcolor}

\newcommand{\e}{\varepsilon}
\newcommand{\iy}{\infty}
\newcommand{\st}{\ : \ }

\renewcommand{\leq}{\leqslant}
\renewcommand{\geq}{\geqslant}

\newcommand{\gL}{\mathsf{L}}

\newcommand{\R}{\mathbf{R}}
\newcommand{\Rp}{\mathbf{R}_+}
\newcommand{\C}{\mathcal{C}}
\newcommand{\CC}{\mathscr{C}}

\newcommand{\Id}{\mathrm{Id}}
\newcommand{\Herm}{\mathbb{H}}
\newcommand{\PSD}{\mathrm{PSD}}

\DeclareMathOperator{\conv}{\mathrm{conv}}

\DeclareMathOperator{\card}{\mathrm{card}}

\DeclareMathOperator{\mathspan}{\mathrm{span}}

\DeclareMathOperator{\aff}{\mathrm{aff}}
\DeclareMathOperator{\inter}{int}
\DeclareMathOperator{\relint}{relint}
\DeclareMathOperator{\relbd}{relbd}

\usepackage{dsfont}

\newcommand{\tmin}{\varodot}
\newcommand{\tmax}{\varoast}

\newcommand{\scalar}[2]{\langle #1 , #2\rangle}

\DeclareMathOperator{\cone}{\mathrm{cone}}
\DeclareMathOperator{\Ext}{\mathrm{Ext}}

\theoremstyle{plain}
\newtheorem{theorem}{Theorem}

\newtheorem{proposition}{Proposition}
\newtheorem{lemma}[proposition]{Lemma}

\newtheorem{corollary}[proposition]{Corollary}

\newtheorem{result}[proposition]{Result}
\theoremstyle{definition}
\newtheorem*{definition}{Definition}
\newtheorem{claim}[proposition]{Claim}
\newtheorem{remark}[proposition]{Remark}

\makeatletter
\def\thmhead@plain#1#2#3{%
  \thmname{#1}\thmnumber{\@ifnotempty{#1}{ }\@upn{#2}}%
  \thmnote{ {\the\thm@notefont#3}}}
\let\thmhead\thmhead@plain
\makeatother

\newcommand{\kite}{\mathsf{Q}}
\newcommand{\bsquare}{\mathsf{S}_b}
\newcommand{\rays}{t}
\newcommand{\rayS}{T}
\newcommand{\rayt}{u}
\newcommand{\rayT}{U}

\begin{document}

\title{Entangleability  of cones}
\author{Guillaume Aubrun}
\address{Institut Camille Jordan, Universit\'e Claude Bernard Lyon 1, 43 boulevard du 11 novembre 1918, 69622 Villeurbanne CEDEX, France}
\email{aubrun@math.univ-lyon1.fr}

\author{Ludovico Lami}
\address{School of Mathematical Sciences and Centre for the Mathematics and Theoretical Physics of Quantum Non-Equilibrium Systems, University of Nottingham, University Park, Nottingham NG7 2RD, United Kingdom}
\address{Institute of Theoretical Physics and IQST, Universit\"{a}t Ulm, Albert-Einstein-Allee 11D-89069 Ulm, Germany}
\email{ludovico.lami@gmail.com}

\author{Carlos Palazuelos}
\address{Dpto. An\'alisis Matem\'atico y Matem\'atica Aplicada, Fac. Ciencias Matemáticas, Universidad Complutense de Madrid, Plaza de Ciencias s/n 28040 Madrid, Spain \& Instituto de Ciencias Matem\'aticas, C/ Nicol\'as Cabrera, 13-15, 28049 Madrid, Spain}
\email{carlospalazuelos@mat.ucm.es}

\author{Martin Pl\'avala}
\address{Mathematical Institute, Slovak Academy of Sciences, \v{S}tef\'{a}nikova 49, Bratislava, Slovakia}
\address{Naturwissenschaftlich-Technische  Fakult\"{a}t, Universit\"{a}t Siegen, 57068 Siegen, Germany}
\email{martin.plavala@uni-siegen.de}

\date{\today}

\begin{abstract}
We solve a long-standing conjecture by Barker, proving that the minimal and maximal tensor products of two finite-dimensional proper cones coincide if and only if one of the two cones is generated by a linearly independent set. 
Here, given two proper cones $\C_1$, $\C_2$, their minimal tensor product is the cone generated by products of the form $x_1\otimes x_2$, where $x_1\in \C_1$ and $x_2\in \C_2$, while their maximal tensor product is the set of tensors that are positive under all product functionals $\varphi_1\otimes \varphi_2$, where $\varphi_1|_{\C_1}\geq 0$ and $\varphi_2|_{\C_2}\geq 0$. 
Our proof techniques involve a mix of convex geometry, elementary algebraic topology, and computations inspired by quantum information theory. Our motivation comes from the foundations of physics: as an application, we show that any two non-classical systems modelled by general probabilistic theories can be entangled.
\end{abstract}

\maketitle

\section{Introduction}

Cones are central objects in various areas of pure and applied mathematics, such as linear algebra, optimisation, convex geometry, differential equations or dynamical systems. Duality usually plays an important role: a convex cone can be either described from the inside (as the set of positive linear combinations of some family of generators) or from the outside (as the set of vectors satisfying some family of linear inequalities). 

When studying linear maps between cones (so, positive operators), tensor products appear naturally. Given two finite-dimensional cones $\C_1$ and $\C_2$, one may define `from the inside' their minimal tensor product $\C_1 \tmin \C_2$, or `from the outside' their maximal tensor product $\C_1 \tmax \C_2$, in such a way that $\C_1 \tmin \C_2\subseteq \C_1 \tmax \C_2$. In formulae, we have
\[ \C_1 \tmin \C_2 \coloneqq \conv \{ x_1 \otimes x_2 \st x_1 \in \C_1, \ x_2 \in \C_2 \} \ \ \textnormal{and} \ \  \C_1 \tmax \C_2 \coloneqq (\C_1^* \tmin \C_2^*)^* ,\]
where $\C^*$ is the dual cone to a cone $\C$. We restrict to proper cones (a closed convex cone $\C$ in a  finite-dimensional real vector space $V$ is \emph{proper} it if satisfies both $\C \cap (-\C) = \{0\}$ and $\C - \C = V$).

In this paper we face the following fundamental question: given a pair of cones $(\C_1,\C_2)$, decide whether $\C_1 \tmin \C_2 = \C_1 \tmax \C_2$ or $\C_1 \tmin \C_2 \subsetneq \C_1 \tmax \C_2$. For reasons which we explain below, we say that the pair $(\C_1,\C_2)$ is \emph{nuclear} in the first case, and \emph{entangleable} in the second case. This question dates back to the work of Barker and Namioka--Phelps in the 1970s.

Our main result provides a simple characterisation of nuclearity, which was conjectured 40 years ago by Barker~\cite{Barker76,Barker81}: a pair $(\C_1,\C_2)$ is nuclear if and only if either $\C_1$ or $\C_2$ is classical. By a \emph{classical} cone we mean a cone isomorphic to $\R_+^n$, or equivalently a cone whose bases are simplices. A famous result by Namioka and Phelps~\cite{NamiokaPhelps69} (see also~\cite{Effros72}) states that if $\C_\square$ denotes a $3$-dimensional cone with $4$ extreme rays (all such cones are isomorphic) and $\C$ is any proper cone, then $\C \tmin \C_\square= \C \tmax \C_\square$ if and only if $\C$ is classical. Note that, according to our main result, the same statement is true if $\C_\square$ is replaced by any non-classical cone.

The case of cones with a centrally symmetric base deserves some attention. If $\C_1$ and $\C_2$ are such cones, their bases can be seen as the unit balls of normed spaces $X_1$ and $X_2$. One checks (see~\cite[Proposition 2.25]{lamiatesi} for a precise statement) that $\C_1 \tmin \C_2$ and $\C_1 \tmax \C_2$ can be related respectively to the projective and injective norms on $X_1 \otimes X_2$. Therefore, when specialised to cones with a centrally symmetric base, our main result is equivalent to the fact that the projective and injective norms are distinct on $X_1 \otimes X_2$ whenever $\dim(X_i) \geq 2$ (for stronger results in this direction, see~\cite{XOR}). Further applications of this connection to the problem of measurement compatibility have been studied in~\cite{Bluhm2020}.

We use the terminology `nuclear' by analogy with the case of $C^*$-algebras. Recall that a pair $(A,B)$ of $C^*$-algebras is a nuclear pair if $A \otimes_{\min} B = A \otimes_{\max} B$, and that a single $C^*$-algebra  $A$ is nuclear if $(A,B)$ is a nuclear pair for every $C^*$-algebra $B$~\cite[Chapter 9]{Pisier20}. Our main result may look surprising to readers familiar with $C^*$-algebras, since the $C^*$-algebraic version does not hold. Indeed, there are examples of non-nuclear $C^*$-algebras $A$, $B$ such that $(A,B)$ is a nuclear pair, the most famous being $A=C^*(\mathbb{F}_\infty)$, $B=B(\ell_2)$, due to Kirchberg~\cite{Kirchberg93}. Remarkably, the analogue of our result becomes true if we restrict to von Neumann algebras~\cite[Theorem 18.13]{Pisier20}. 

For cones of positive semi-definite matrices, the fact that the minimal and maximal tensor products do not coincide is intimately connected to the phenomenon of quantum entanglement. This observation explains our terminology `entangleable'. Moreover, although we hardly mention them in the present paper, the question we study has very strong motivations stemming from the foundations of physics, and more precisely from the study of general probabilistic theories (GPTs), a framework based on convex geometry which encompasses both classical probability and quantum physics. In this context, our result implies that -- under a natural no-restriction hypothesis -- any two non-classical GPTs exhibit some form of entanglement when combined, either at the level of states or at that of measurements~\cite{ALPP} (see also~\cite[Chapter~2]{lamiatesi}).

Our proof of Barker's conjecture goes as follows: we exhibit a geometric property, the kite-square sandwiching, which we prove to characterise precisely non-classical cones. This geometric property involves cones based on two specific planar shapes: the \emph{kite} and the \emph{blunt square}. We then show that kite-square sandwichings can be used to produce a certificate of entangleability. Our methods involve convex geometry, elementary algebraic topology and explicit computations which are inspired by quantum information theory.

We restrict ourselves to finite-dimensional cones in the present paper. One may define tensor products, nuclearity and entangleability for infinite-dimensional cones, for example using the language of function systems as in~\cite{Effros72,Han16}. However, since most of the tools we use, either topological or geometric, are inherently finite-dimensional, the study of infinite-dimensional cones will probably require a different approach.

As a byproduct of our main result, we answer a question raised in the study of matrix convex sets~\cite{PSS18}, which happens to be a particular case of Barker's conjecture (see Corollary~\ref{corollary:OS} below).

\section{Notation and statement of the main results} 

\subsection{Convex cones}

Throughout the paper, all the vector spaces are assumed to be finite-di\-men\-sion\-al and over the real field. We denote vector spaces by symbols such as $V$, $V_1$ or $V'$. A subset $\C$ of a vector space $V$ is a \emph{convex cone}, or simply a \emph{cone}, if it satisfies $sx+ty \in \C$ for every $x$, $y \in \C$ and $s$, $t \in \Rp$ (we denote by $\Rp$ the half-line $[0,\iy)$). We denote by $\cone(A)$ the cone generated by a subset $A \subset V$. 

A cone $\C \subset V$ is said to be \emph{generating} if it spans $V$ as a vector space, or equivalently if $\C - \C = V$. Also, $\C$ is said to be \emph{salient} (also called \emph{pointed}) if it does not contain a line, or equivalently if $\C \cap (-\C) = \{0\}$. Finally, $\C$ is said to be \emph{proper} if it is closed, salient and generating. 

A \emph{convex body} is a compact convex subset of a vector space with nonempty interior. We denote respectively by $\inter(K)$ and $\partial K$ the interior and boundary of a convex body $K$. If $K \subset V$ is a convex set, then the \emph{cone over} $K$ is the cone in $V \times \R$ defined as
\[ \CC(K) = \cone( K \times \{1\} ) = \{ (x \,;\,t) \in V\times \Rp \st x \in tK \}. \]
If $K$ is a convex body, then $\CC(K)$ is easily shown to be a proper cone.

Two cones $\C$ and $\C'$, living in vector spaces $V$ and $V'$, are called \emph{isomorphic} if there is a linear bijection $\Phi : V \to V'$ such that $\Phi(\C) = \C'$. We use repeatedly the following elementary fact: if $\C$ is a proper cone, then there is a convex body $K$ in $\R^{\dim (\C) -1}$ such that $\C$ is isomorphic to $\CC(K)$. 

Let $V$ be a vector space, and  $V^*$ its dual space. If $\C$ is a cone in $V$, its \emph{dual cone} is defined as
\[ \C^* = \{ f \in V^* \st f(x) \geq 0 \textnormal{ for every } x \in \C \} .\]
The bipolar theorem~\cite[Theorem~14.1]{Rockafellar70} asserts that for a closed cone $\C$, we have $\C = (\C^*)^*$ when identifying $V$ with the bidual $V^{**}$.

Let $\C$ be a cone. An element $x \in \C$ is an \emph{extreme ray generator} if the equation $x=y+z$ for $y$, $z \in \C$ implies $y=\alpha x$ for some $\alpha\in [0,1]$. In that case, the set $\{tx \st t\in\Rp\}$ is called an \emph{extreme ray} of $\C$.

\subsection{Entangleability of cones}

How to define the tensor product of two cones? It has been realised by several authors~\cite{Barker81, Birnbaum76, Mulansky97, NamiokaPhelps69} that there are at least two meaningful answers, since one may define naturally the minimal and the maximal tensor product of two cones. These objects, which are sometimes called the projective and injective tensor products, are dual to each other. We now introduce them.

Let $V_1$, $V_2$ be vector spaces, and $\C_1 \subset V_1$, $\C_2 \subset V_2$ be convex cones.
We define the \emph{minimal tensor product} of $\C_1$ and $\C_2$ as
\begin{equation} \label{eq:def-tmin} \C_1 \tmin \C_2 \coloneqq \conv \{ x_1 \otimes x_2 \st x_1 \in \C_1, \ x_2 \in \C_2 \}, \end{equation}
and the \emph{maximal tensor product} of $\C_1$ and $\C_2$ as
\begin{equation} \label{eq:def-tmax} \C_1 \tmax \C_2 \coloneqq \{ z \in V_1 \otimes V_2 \st (f_1 \otimes f_2)(z) \geq 0 \textnormal{ for every } f_1 \in \C_1^*,\ f_2 \in \C_2^* \} .\end{equation}
It is trivial to check that the inclusion $\C_1 \tmin \C_2 \subset \C_1 \tmax \C_2$ holds always true. Following Barker, `a  major  open question is to determine necessary and sufficient conditions for equality to hold'~\cite[p.~197]{Barker76}. Our paper answers this question.

By definition, we have that $\C_1 \tmax \C_2 = (\C_1^* \tmin \C_2^*)^*$, where we identify $V_1 \otimes V_2$ with $(V_1^* \otimes V_2^*)^*$. If $\C_1$ and $\C_2$ are proper, then $\C_1 \tmin \C_2$ is proper as well (see {\cite[Exercise 4.14]{ABMB}}) and the bipolar theorem implies that $\C_1 \tmin \C_2 = (\C_1^* \tmax \C_2^*)^*$: the minimal and maximal tensor products are dual to each other.

Let $\C_1$, $\C_2$ be two proper cones. We say that the pair $(\C_1,\C_2)$ is \emph{nuclear} if $\C_1 \tmin \C_2 = \C_1 \tmax \C_2$, and that $(\C_1,\C_2)$ is \emph{entangleable} if $\C_1 \tmin \C_2 \neq \C_1 \tmax \C_2$. The terminology `nuclear' is borrowed from the analogous notion in $C^*$-algebras, while the concept of entangleability comes from the interpretation of cones in the context of general probabilistic theories (GPTs), of which quantum mechanics is a special case (see~\cite{ALPP} for a thorough discussion of these ideas). Cones corresponding to quantum mechanics belong to the family $(\PSD_n)_{n \geq 1}$, where $\PSD_n$ denotes the cone of $n \times n$ positive semi-definite matrices with complex entries. The phenomenon of quantum entanglement is connected with the fact that $\PSD_m \tmin \PSD_n \neq \PSD_m \tmax \PSD_n$ for $m$, $n \geq 2$, and therefore $(\PSD_m,\PSD_n)$ is a fundamental example of an entangleable pair.

A cone $\C$ is said to be \emph{classical} if it is isomorphic to $\R_+^d$ for $d= \dim (\C)$. (This terminology comes from the fact that $\R_+^d$ corresponds to classical probability theory on an alphabet of size $d$ in the GPT formalism. Alternative terms such as `simplicial cone', `minihedral cone' or `lattice cone' are used throughout the literature). Equivalently, a cone $\C$ in a vector space $V$ is classical if and only if there is a basis $A$ of $V$ (as a vector space) such that $\C=\cone(A)$. It was noticed early~\cite{NamiokaPhelps69} that a pair $(\C_1,\C_2)$ of proper cones is nuclear whenever either $\C_1$ or $\C_2$ is classical, and a natural conjecture, implicit in~\cite{Barker76} and explicit in~\cite{Barker81}, is that the converse holds. We prove this conjecture, giving a complete understanding of the entangleability of cones.

\begin{theorem} \label{theorem:main}
Let $\C_1$ and $\C_2$ be proper cones. Then $(\C_1,\C_2)$ is nuclear if and only if $\C_1$ or $\C_2$ is classical.
\end{theorem}

Special cases of Theorem~\ref{theorem:main} were known prior to this paper. The easiest statement to prove is the fact that a pair of the form $(\C,\C^*)$ is nuclear if and only if $\C$ is classical; this was observed in~\cite{BarkerLoewy75,Tam77} and is also equivalent to the no-broadcasting theorem in GPTs~\cite{Barnum2007}. Another special case of Theorem~\ref{theorem:main} is the following result by Namioka and Phelps~\cite{NamiokaPhelps69}: if $\C_\square$ denotes a $3$-dimensional cone with $4$ extreme rays (all such cones are isomorphic) and $\C$ is any proper cone, then the pair $(\C,\C_\square)$ is nuclear if and only if $\C$ is classical. Note that, according to Theorem~\ref{theorem:main}, one can replace $\C_\square$ by any non-classical cone in the previous statement. 

We emphasise that the present paper is a study of nuclearity of \emph{pairs of cones}. According to conventional functional-analytic terminology, one may call a single proper cone $\C$ nuclear if $\C \tmin \C' =\C \tmax \C'$ for every proper cone $\C'$. However this notion of nuclearity is well-understood: it is a immediate consequence of the aforementioned result by Namioka and Phelps that a single cone is nuclear if and only if it is classical. Determining which are the nuclear pairs of cones is more challenging and is the point of our paper.

Prior to this work, partial results have been obtained recently in~\cite{ALP19} (see also~\cite{debruyn2020tensor}), where it is proved that (1) Theorem~\ref{theorem:main} holds if $\dim(\C_1)=\dim(\C_2)=3$, and (2) Theorem~\ref{theorem:main} holds if $\C_1$ and $\C_2$ are polyhedral cones.

\subsection{Consequences of Theorem~\ref{theorem:main}} 

\subsubsection{Cones of positive maps}

Let $V_1$, $V_2$ be finite-dimensional vector spaces. The tensor product $V_1 \otimes V_2$ is canonically isomorphic to the space $\gL(V_1^*,V_2)$ of linear operators from $V_1^*$ to $V_2$. Consider now proper cones $\C_1 \subset V_1$ and $\C_2 \subset V_2$. Under the isomorphism mentioned above, the maximal tensor product $\C_1 \tmax \C_2$ corresponds to the cone of maps $\Phi \in \gL(V_1^*,V_2)$ which are $(\C_1^*,\C_2)$-\emph{positive}, i.e.\ such that $\Phi(\C_1^*) \subset \C_2$. Similarly, the minimal tensor product $\C_1 \tmin \C_2$ corresponds to the cone generated by  $(\C_1^*,\C_2)$-positive maps of rank $1$. We obtain therefore the following restatement of Theorem~\ref{theorem:main} (to show the equivalence between both statements, remember that $\C_1$ is classical if and only if the dual cone $\C_1^*$ is classical).

\begin{corollary}
Let $\C_1$ and $\C_2$ be proper cones. The following are equivalent
\begin{enumerate}
    \item Every $(\C_1,\C_2)$-positive map is a sum of $(\C_1,\C_2)$-positive maps of rank $1$.
    \item Either $\C_1$ or $\C_2$ is classical.
\end{enumerate}
\end{corollary}

\subsubsection{Matrix convex sets and operator systems}

As a consequence of Theorem~\ref{theorem:main}, we answer a question raised in~\cite{PSS18} about maximal and minimal matrix convex sets, or equivalently about maximal and minimal operator systems. We here use the language of operator systems and refer to~\cite[\S 7.1]{PSS18} for the translation in terms of matrix convex sets. As explained in~\cite{FNT17}, an (abstract) operator system in $d$ variables can be described by a sequence $(\C_n)_{n \geq 1}$ of proper cones, where $\C_n$ lives in the space $\Herm_n^d$ of $d$-tuples of $n \times n$ Hermitian matrices, with the property that for every $m \times n$ matrix $B$,
\[ (A_1,\dots,A_d) \in \C_n \Longrightarrow (BA_1B^\dagger, \dots, BA_dB^\dagger) \in \C_m. \] 
The usual definition of an operator system also requires to specify an order unit (=an interior point) for each cone $\C_n$. We ignore this condition since the choice of an order unit is irrelevant for our purposes (see~\cite[Remark 1.2(c)]{FNT17}). 
As it turns out, given a proper cone $\C \subset \R^d$, there is a minimal operator system $(\C_n^{\min})_{n \geq 1}$ and a maximal operator system $(\C_n^{\max})_{n \geq 1}$ satisfying the condition $\C_1^{\min} = \C_1^{\max} = \C$. This means that any operator system $(\C_n)_{n \geq 1}$ such that $\C_1=\C$ must satisfy $\C_n^{\min} \subset \C_n \subset \C_n^{\max}$. (We warn the reader that our use of the terminology `minimal' and `maximal' for tensor products, which follows~\cite{FNT17}, is reversed with respect to the common practice in functional analysis.) Moreover, the minimal and maximal operator systems can be described as
\begin{gather*} \C_n^{\min} = \C \tmin \PSD_n, \\
\C_n^{\max} = \C \tmax \PSD_n,
\end{gather*}
where $\PSD_n \subset \Herm_n$ is the cone of positive semidefinite matrices, and with the identification of $\R^d \otimes \Herm_n$ with $\Herm_n^d$. Our result is the following
\begin{corollary} \label{corollary:OS}
Let $\C$ be a proper cone, and $n \geq 2$. Then $\C_n^{\min} = \C_n^{\max}$ if and only if $\C$ is classical.
\end{corollary}
To deduce Corollary~\ref{corollary:OS} from Theorem~\ref{theorem:main}, it suffices to notice that the cone $\PSD_n$ is not classical for $n \geq 2$. Corollary~\ref{corollary:OS} improves on results from~\cite{PSS18} (where the same result was proved under the condition $\log (n)=\Omega(\dim(\C))$, answering in particular~\cite[Problem 4.3]{PSS18} in the optimal way, and from~\cite{HuberNetzer18} (where the same result was proved under the assumption that $\C$ is polyhedral).

\subsubsection{General probabilistic theories}

Our motivation for the study of entangleability of cones originates from the foundations of physics. Theorem~\ref{theorem:main} can be reformulated within the framework of general probabilistic theories (GPTs) as follows. 

\begin{result}
All pairs of non-classical GPTs can be entangled.
\end{result}

We state this result informally on purpose, and refer the interested reader to~\cite{ALPP}, where the terminology is introduced, and consequences for the foundations of physics are thoroughly discussed.

\subsubsection{More than 2 cones}

It is straightforward to define the maximal and minimal tensor product of $k \geq 2$ cones by extending formulae~\eqref{eq:def-tmin} and~\eqref{eq:def-tmax} to $k$-fold tensors. We obtain easily the following generalisation of Theorem~\ref{theorem:main}.

\begin{corollary}
Let $k \geq 2$ and $\C_1, \dots, \C_k$ be proper cones. Then the following are equivalent
\begin{enumerate}
    \item We have $\C_1 \tmin \cdots \tmin \C_k = \C_1 \tmax \cdots \tmax \C_k$,
    \item At most one among the cones $\C_1,\dots,\C_k$ is non-classical.
\end{enumerate}
\end{corollary}

\begin{proof}
The implication (2) $\Longrightarrow$ (1) is by induction on $k$ using the easy part of Theorem~\ref{theorem:main}. Conversely, assuming (1), we prove that for every $i \neq j$, either $\C_i$ or $\C_j$ is classical. Without loss of generality, assume $(i,j)=(1,2)$. Fix nonzero elements $f_3 \in \C_3^*, \cdots, f_k \in \C_k^*$. Denoting by $V_i$ the ambient space where $\C_i$ lives, one checks that
\begin{gather*}
(\Id_{V_1 \otimes V_2} \otimes f_3 \otimes \cdots \otimes f_k)(\C_1 \tmin \C_2 \tmin \cdots \tmin \C_k) = \C_1 \tmin \C_2 \\
(\Id_{V_1 \otimes V_2} \otimes f_3 \otimes \cdots \otimes f_k)(\C_1 \tmax \C_2 \tmax \cdots \tmax \C_k) = \C_1 \tmax \C_2.
\end{gather*}
Our hypothesis implies that $(\C_1,\C_2)$ is nuclear, and the result follows by Theorem~\ref{theorem:main}.
\end{proof}

\subsection{Sketch of proof and organisation of the paper}

Before we describe our argument, we present a short overview of it in which we use the language of quantum information theory, as this may be profitable to some of our readers; others may skip this paragraph. Given two non-classical cones~$\C_1$ and $\C_2$, we need to construct a vector $\omega$ in $\C_1 \tmax \C_2$ which is \emph{entangled}, i.e.\ not in $\C_1 \tmin \C_2$. The simplest non-classical cones are the 3-dimensional cones generated by 4 points in convex position, which we call kites. Our vector $\omega$ is constructed from a pair of kites embedded in $\C_i$ via an  explicit formula reminiscent of Popescu--Rohrlich boxes~\cite{PR-boxes}. In order to certify that $\omega$ is entangled, we show that it violates an inequality based on the simplest of all Bell inequalities: the Clauser--Horne--Shimony--Holt (CHSH) inequality~\cite{CHSH}. 
The CHSH inequality involves two binary measurements which we encode by mapping each cone $\C_i$ inside the cone over a square. As we need to investigate the equality case in the CHSH inequality, we consider instead a blunt version of the square. For this strategy to be successful, each of the cones $\C_i$ must be related to both a kite and a blunt square in a compatible way. We conclude by showing that this situation occurs for every non-classical cone.

Before delving into the details of our argument, we need to fix some terminology. We start by defining particular planar convex shapes. First, the \emph{blunt square} is defined to be a square minus its vertices
\[ \bsquare \coloneqq [-1,1]^2 \setminus \{-1,1\}^2. \]
The fact that we look at the blunt square instead of the usual square is critical to our arguments, as the example given in Remark~\ref{remark:why-blunt} will show.
\begin{figure}[htbp] \begin{center}
\begin{tikzpicture}[scale=2]
	\coordinate (a) at (1,1);
	\coordinate (b) at (1,-1);
	\coordinate (c) at (-1,-1);
	\coordinate (d) at (-1,1) ;
	\coordinate (A) at (0.7,1);
	\coordinate (B) at (1,0.9);
	\coordinate (C) at (-0.2,-1);
	\coordinate (D) at (-1,0.5) ;
	\draw[fill=gray!30] (a)--(b)--(c)--(d)--(a) ;
	\draw[fill=gray!90] (A)--(B)--(C)--(D)--(A) ;
	\draw[gray!0] (a) node {$\bullet$};
        \draw[gray!0] (b) node {$\bullet$};
        \draw[gray!0] (c) node {$\bullet$};
        \draw[gray!0] (d) node {$\bullet$};
	\draw (-0.1,-0.1) node {$\kite_{\alpha}$};
	\draw (0.6,-0.5) node {$\bsquare$};
        \end{tikzpicture} \end{center}
\caption{A kite inside the blunt square}
\end{figure}
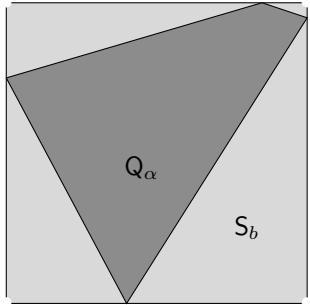
Note that $\bsquare$ is neither closed nor open, and that the same is true for the cone $\CC(\bsquare)$. We then define a \emph{kite} to be a convex body of the form
\[ \kite_\alpha = \conv \{ (1,\alpha_1), (\alpha_2,1), (-1,\alpha_3), (\alpha_4,-1) \} , \]
where $\alpha=(\alpha_1,\dots,\alpha_4) \in (-1,1)^4$. Note that any kite is contained in the blunt square.

We now introduce the main geometric tool used in the proof of Theorem~\ref{theorem:main}. Let $\C$ be a proper cone in a finite-dimensional vector space $V$. We say that $\C$ \emph{admits a kite-square sandwiching} if there is a kite $\kite_\alpha$ and two linear maps $\Psi : \R^3 \to V$, $\Phi : V \to \R^3$ such that $\Phi \circ \Psi = \Id$, $\Psi(\CC(\kite_\alpha)) \subset \C$ and $\Phi(\C) \subset \CC(\bsquare)$. We prove that this property is shared by all non-classical proper cones (note that any non-classical cone $\C$ must satisfy $\dim(\C)\geq 3$). 

\begin{theorem} \label{theorem:KS}
Let $\C$ be a proper cone. Then $\C$ is non-classical if and only if it admits a kite-square sandwiching.
\end{theorem}

Our second step is to deduce entangleability from the existence of kite-square sandwichings.

\begin{theorem} \label{theorem:KS-entangleable}
Let $\C_1$ and $\C_2$ be proper cones, both admitting a kite-square sandwiching. Then $(\C_1,\C_2)$ is entangleable.
\end{theorem}

It is then immediate to prove Theorem~\ref{theorem:main}. The fact that $(\C_1,\C_2)$ is nuclear whenever either $\C_1$ or $\C_2$ is classical is well known and is the easy direction; for the benefit of the reader we include a proof below. The fact that $(\C_1,\C_2)$ is entangleable whenever both $\C_1$ and $\C_2$ are non-classical is an immediate consequence of Theorems~\ref{theorem:KS} and~\ref{theorem:KS-entangleable}.

\begin{proof}[Proof of the easy direction in Theorem~\ref{theorem:main}]
Let $\C_1 \subset V_1$ and $\C_2 \subset V_2$ be proper cones, and assume that one of them (say, $\C_1$) is classical. Let $(e_i)$ be a basis of $V_1$ such that $\C_1 = \cone (e_i)$. Consider the dual basis $(e^*_i)$ of $V_1^*$, which satisfies $e_i^*(e_j)=\delta_{i,j}$.
Decompose an arbitrary $z \in \C_1 \tmax \C_2$ as $z = \sum e_i \otimes x_i$, where $x_i\in V_2$. By definition of maximal tensor product, for every $f \in \C_2^*$ we have that $0 \leq (e_i^* \otimes f)(z) = f(x_i)$ for all $i$. This shows that $x_i \in \C_2^{**} = \C_2$. Hence, $z \in \C_1 \tmin \C_2$ and consequently $\C_1\tmax \C_2 = \C_1 \tmin \C_2$.
\end{proof}

Our paper is organised as follows. Section~\ref{section:KS-entagleable} is devoted to the proof of Theorem~\ref{theorem:KS-entangleable}. It is based on explicit computations on kites and blunt squares. Section~\ref{section:lemmas} gathers several lemmata which are used in the proof of Theorem~\ref{theorem:KS}. The proof of Theorem~\ref{theorem:KS} is relegated to Section~\ref{section:main-proof}. 

\section{Proof of Theorem~\ref{theorem:KS-entangleable}} \label{section:KS-entagleable}

In this section we prove that any two cones that both admit a kite-square sandwiching form an entangleable pair. We first observe that the maximal tensor product of two cones over kites sticks out from the minimal tensor product of cones over the blunt square.

\begin{proposition} \label{prop:g-kite}
Fix $\alpha$ and $\beta \in (-1,1)^4$. Then $\CC(\kite_\alpha) \tmax \CC(\kite_\beta)$ is not a subset of 
$\CC(\bsquare) \tmin \CC(\bsquare)$. In other words, there is $\omega \in \CC(\kite_\alpha) \tmax \CC(\kite_\beta)$ such that $\omega \not \in \CC(\bsquare) \tmin \CC(\bsquare)$.
\end{proposition}

Assuming Proposition~\ref{prop:g-kite} for the moment, it is easy to deduce Theorem~\ref{theorem:KS-entangleable}.

\begin{proof}[Proof of Theorem~\ref{theorem:KS-entangleable}]
Let $\Phi_1$, $\Psi_1$, $(\alpha_i)_{1 \leq i \leq 4}$ and $\Phi_2$, $\Psi_2$, $(\beta_i)_{1 \leq i \leq 4}$ as in the definition of a kite-square sandwiching, for $\C_1$ and $\C_2$ respectively. Then, for every $z \in \CC(\kite_\alpha) \tmax \CC(\kite_\beta)$, we have that $(\Psi_1 \otimes \Psi_2)(z) \in \C_1 \tmax \C_2$ (this is because $f \in \C_1^*$ implies that $f \circ \Psi_1 \in \CC(\kite_\alpha)^*$, and analogously for $\C_2$). If we assume by contradiction that $(\C_1,\C_2)$ is nuclear, then there is a decomposition 
\[ (\Psi_1 \otimes \Psi_2)(z) = \sum x_k \otimes y_k \]
with $x_k \in \C_1$, $y_k \in \C_2$, and therefore
\[ z = (\Phi_1 \otimes \Phi_2)(\Psi_1 \otimes \Psi_2)(z) = \sum \Phi_1(x_k) \otimes \Phi_2(y_k) \in \CC(\bsquare) \tmin \CC(\bsquare), \]
contradicting Proposition~\ref{prop:g-kite}.
\end{proof}

The proof of Proposition~\ref{prop:g-kite} relies on the following lemma, which is used to construct non-trivial elements in the maximal tensor product of two cones.

\begin{lemma} \label{lemma:omega-tmax}
Let $\C$, $\C'$ be two proper cones. Let $\rayS_1$, $\rayS_2$, $\rayS_3$, $\rayS_4$ be elements of $\C$ satisfying $\rayS_1+\rayS_3=\rayS_2+\rayS_4$, and $\rayT_1$, $\rayT_2$, $\rayT_3$, $\rayT_4$ be elements of $\C'$ satisfying $\rayT_1+\rayT_3=\rayT_2+\rayT_4$. Then
\begin{equation} \label{eq:omega} \omega \coloneqq \rayS_1 \otimes \rayT_2 - \rayS_2 \otimes \rayT_2 + \rayS_2 \otimes \rayT_1 + \rayS_3 \otimes \rayT_3 \in \C \tmax \C'. \end{equation}
\end{lemma}

\begin{proof}
By definition of the maximal tensor product, we need to check that for every $\varphi \in \C^*$, $\varphi' \in (\C')^*$, we have that $(\varphi \otimes \varphi')(\omega) \geq 0$. Set $x_i=\varphi(\rayS_i)$ and $y_i=\varphi'(\rayT_i)$, so that $x_i \geq 0$, $y_i \geq 0$. We compute
\begin{align*}
(\varphi \otimes \varphi')(\omega) & =  x_1y_2-x_2y_2+x_2y_1+x_3y_3 \\
& =  x_1y_1+x_3y_3 - (x_2-x_1)(y_2-y_1).
\end{align*}
Since $-x_1 \leq x_2-x_1  \leq x_3$ and $-y_1 \leq y_2-y_1 \leq y_3$, we have that $(x_2-x_1)(y_2-y_1) \leq \max(x_1y_1,x_3y_3)$, implying that $(\varphi \otimes \varphi')(\omega) \geq 0$. This completes the proof.
\end{proof}

\begin{proof}[Proof of Proposition~\ref{prop:g-kite}]
Let $\alpha$, $\beta \in (-1,1)^4$. The extreme rays of $\CC(\kite_\alpha)$ and $\CC(\kite_\beta)$ are generated respectively by the vectors
\begin{gather*}
\rays_1 = (1,\alpha_1 \,;\, 1), \ \rays_2=(\alpha_2,1\,;\, 1),\  \rays_3=(-1,\alpha_3\,;\,1),\ \rays_4=(\alpha_4,-1\,;\,1) , \\
 \rayt_1 = (1,\beta_1 \,;\, 1), \ \rayt_2=(\beta_2,1\,;\, 1),\  \rayt_3=(-1,\beta_3\,;\,1),\ \rayt_4=(\beta_4,-1\,;\,1) . 
\end{gather*} 
In order to use Lemma~\ref{lemma:omega-tmax}, we are going to replace $(\rays_i)$ and $(\rayt_i)$ by suitable positive multiples, denoted by $(\rayS_i)$ and $(\rayT_i)$ and defined below, which have the property that $\rayS_1+\rayS_3=\rayS_2+\rayS_4$ and $\rayT_1+\rayT_3=\rayT_2+\rayT_4$. We set
\begin{gather*}
\rayS_1 = (2+\alpha_2+\alpha_4+\alpha_3(\alpha_2-\alpha_4)) \cdot \rays_1, \\
\rayS_2 = (2+\alpha_1+\alpha_3+\alpha_4(\alpha_1-\alpha_3)) \cdot \rays_2, \\ 
\rayS_3 = (2-\alpha_4-\alpha_2+\alpha_1(\alpha_4-\alpha_2)) \cdot \rays_3, \\
\rayS_4 = (2-\alpha_3-\alpha_1+\alpha_2(\alpha_3-\alpha_1)) \cdot \rays_4, \\
\rayT_1 = (2+\beta_2+\beta_4+\beta_3(\beta_2-\beta_4)) \cdot \rayt_1, \\
\rayT_2 = (2+\beta_1+\beta_3+\beta_4(\beta_1-\beta_3)) \cdot \rayt_2, \\ 
\rayT_3 = (2-\beta_4-\beta_2+\beta_1(\beta_4-\beta_2)) \cdot \rayt_3, \\
\rayT_4 = (2-\beta_3-\beta_1+\beta_2(\beta_3-\beta_1)) \cdot \rayt_4. 
\end{gather*}
One checks that the proportionality coefficients in the previous $8$ equations are positive. Note that while the form of the previous coefficients may appear ad hoc, there is actually no mystery. For example, the reader may determine the coefficients of $\rayS_i$'s (up to a common multiple) by finding the point of the intersection of the diagonals of the kite $\kite_\alpha$ and representing it as convex combinations of pairs of opposite vertices.

It follows from Lemma~\ref{lemma:omega-tmax} that the element
\begin{equation}
\label{eq:omega-alphabeta}    
 \omega_{\alpha,\beta} \coloneqq 
\rayS_1 \otimes \rayT_2 - \rayS_2 \otimes \rayT_2 + \rayS_2 \otimes \rayT_1 + \rayS_3 \otimes \rayT_3
\end{equation}
belongs to $\CC(\kite_\alpha) \tmax \CC(\kite_\beta)$.
We then introduce a linear form $f$ on $\R^3 \otimes \R^3$ (identified with the space of $3 \times 3$ matrices), defined for $m = (m_{ij})_{1 \leq i,j \leq 3}$ by
\begin{equation} \label{eq:f} f(m) = m_{11} + m_{12} + m_{21} - m_{22} - 2m_{33}. \end{equation}
\begin{lemma} \label{claim:omega-tmin}
It holds that $f(m)<0$ for every nonzero $m \in \CC(\bsquare) \tmin \CC(\bsquare)$.
\end{lemma}

\begin{proof}[Proof of Lemma~\ref{claim:omega-tmin}]
It is enough to prove Lemma~\ref{claim:omega-tmin} for $m = (x,y\,;\,1) \otimes (x',y'\,;\,1)$ with $(x,y) \in \bsquare$, $(x',y') \in \bsquare$, since any element in $\CC(\bsquare) \tmin \CC(\bsquare)$ is a positive linear combination of such tensors. Then, $f(m)= xx'+xy'+yx'-yy'-2$. We are reduced to the following elementary inequality 
\begin{equation} \label{eq:blunt-CHSH} \textnormal{if } (x,y) \in \bsquare \textnormal{ and } (x',y') \in \bsquare, \textnormal{ then } xx' + xy' + yx' - yy' < 2 .  \end{equation}
The inequality~\eqref{eq:blunt-CHSH} is a variant of the CHSH inequality. A quick proof goes as follows:
\begin{align*} | xx' + xy' + yx' - yy' | & \leq |x| \cdot |x'+y'| + |y| \cdot |x'-y'| \\
 & \leq |x'+y'| + |x'-y'| \\
 & \leq 2.
\end{align*}
We argue that one of the inequalities must be strict. Assume the last inequality to be an equality. In this case, either $|x'|$ or $|y'|$ must equal $1$. Since they cannot both equal $1$, the numbers $x'+y'$ and $x'-y'$ are nonzero. Now, if the second inequality is also an equality, then it follows that $|x|=|y|=1$, a contradiction.
\end{proof}

We now combine $\omega_{\alpha,\beta}$ defined by~\eqref{eq:omega-alphabeta} with $f$ defined in~\eqref{eq:f}. A series of straightforward yet somewhat cumbersome computations -- which we postpone -- shows that
\begin{equation} \label{eq:magical} f(\omega_{\alpha,\beta}) = - R(\alpha)R(\beta),\end{equation}
where we denote, for $\gamma=(\gamma_i)_{1 \leq i \leq 4}$,
\[ R(\gamma) = (\gamma_1 \gamma_2 -1)(\gamma_3 - \gamma_4) - (\gamma_3 \gamma_4-1)(\gamma_1-\gamma_2) . \]

Assume that $f(\omega_{\alpha,\beta}) \geq 0$. In this, case, it follows immediately from Lemmata~\ref{lemma:omega-tmax} and $\ref{claim:omega-tmin}$ that the element $\omega_{\alpha,\beta}$ belongs to $\CC(\kite_\alpha) \tmax \CC(\kite_\beta) \setminus \CC(\bsquare) \tmin \CC(\bsquare)$, as claimed. 

If $f(\omega_{\alpha,\beta}) < 0$, we reduce to the previous case by exploiting symmetries of the problem.
For $(x,y\,;\,t) \in \R^3$, we set $\sigma(x,y\,;\,t) \coloneqq (y,x\,;\,t)$. We have $\sigma(\CC(\bsquare)) = \CC(\bsquare)$ and $\sigma(\CC(\kite_\alpha)) = \CC(\kite_{\overline{\alpha}})$ with $\overline{\alpha} = (\alpha_2,\alpha_1,\alpha_4,\alpha_3)$. We also check that $R(\overline{\alpha})=-R(\alpha)$ and therefore $f(\omega_{\overline{\alpha}, \beta}) = -f(\omega_{\alpha,\beta}) > 0$. In this case, we already observed that
\[ \omega_{\overline{\alpha},\beta} \in \CC(\kite_{\overline{\alpha}}) \tmax \CC(\kite_\beta) \setminus \CC(\bsquare) \tmin \CC(\bsquare). \]
Then, the element $(\sigma \otimes \Id)(\omega_{\overline{\alpha},\beta})$ belongs to
\[ (\sigma \otimes \Id)\left( \CC(\kite_{\overline{\alpha}}) \tmax \CC(\kite_\beta) \right) = \CC(\kite_{\alpha}) \tmax \CC(\kite_\beta) \]
and does not belong to 
\[ (\sigma \otimes \Id)\left( \CC(\bsquare) \tmin \CC(\bsquare) \right) = \CC(\bsquare) \tmin \CC(\bsquare) \]
as needed. In the last two equations, we used an easily verified property of the minimal and maximal tensor products: if $\C_1 \subset V_1$ and $\C_2 \subset V_2$ are cones and $\Phi : V_1 \to V'_1$ is an isomorphism, then $(\Phi \otimes \Id)(\C_1 \tmin \C_2) = \Phi(\C_1) \tmin \C_2$ and $(\Phi \otimes \Id)(\C_1 \tmax \C_2) = \Phi(\C_1) \tmax \C_2$.
\end{proof}

We now justify the equality~\eqref{eq:magical}, by brute force.\footnote{Alternatively, the reader will find at \url{https://github.com/gaubrun/entangleability} a SageMath script which checks the correctness of~\eqref{eq:magical}.} We use shortcuts such as $\alpha_{12}=\alpha_1 \alpha_2$, $\beta_{134}=\beta_1 \beta_3 \beta_4$, and so on. Let $(\omega_{ij})_{1 \leq i,j \leq 3}$ be the coordinates of the tensor $\omega_{\alpha,\beta}$. We have

\begin{dmath*}
\omega_{11}  =   
- \alpha_{124} \beta_{124}+ \alpha_{124} \beta_{234}+ \alpha_{234} \beta_{124}- \alpha_{234} \beta_{234}- \alpha_{124} \beta_{12}- \alpha_{124} \beta_{34}- \alpha_{12} \beta_{124}+ \alpha_{12} \beta_{234}   
+ \alpha_{234} \beta_{12}+ \alpha_{234} \beta_{34}- \alpha_{34} \beta_{124}+ \alpha_{34} \beta_{234}- \alpha_{124} \beta_2+ \alpha_{124} \beta_4- \alpha_{12} \beta_{14}- \alpha_{12} \beta_{34}  
- \alpha_{14} \beta_{12} + \alpha_{14} \beta_{14}+ \alpha_{234} \beta_2- \alpha_{234} \beta_4+ \alpha_{23} \beta_{23}- \alpha_{23} \beta_{34}- \alpha_2 \beta_{124}+ \alpha_2 \beta_{234}- \alpha_{34} \beta_{12} 
- \alpha_{34} \beta_{23} + \alpha_4 \beta_{124} - \alpha_4 \beta_{234}+2 \alpha_{124}+2 \alpha_{12} \beta_4- \alpha_{14} \beta_2- \alpha_{14} \beta_4-2 \alpha_{234}+ \alpha_{23} \beta_2  
+ \alpha_{23} \beta_4- \alpha_2 \beta_{14} + \alpha_2 \beta_{23}-2 \alpha_2 \beta_{34}-2 \alpha_{34} \beta_2+2 \alpha_4 \beta_{12}- \alpha_4 \beta_{14}  + \alpha_4 \beta_{23}+2 \beta_{124}-2 \beta_{234} 
+2 \alpha_{14}+2 \alpha_{23} + \alpha_2 \beta_2+3 \alpha_2 \beta_4 + 3 \alpha_4 \beta_2+ \alpha_4 \beta_4+2 \beta_{14}+2 \beta_{23}+2 \alpha_2-2 \alpha_4+2 \beta_2-2 \beta_4+4,    
\end{dmath*}

\begin{dmath*}
\omega_{12} =  \
\alpha_{124} \beta_{123}-\alpha_{124} \beta_{134}-\alpha_{234} \beta_{123}+\alpha_{234} \beta_{134}+\alpha_{124} \beta_{12}+\alpha_{124} \beta_{34}+\alpha_{14} \beta_{123}   
-\alpha_{14} \beta_{134}-\alpha_{234} \beta_{12}-\alpha_{234} \beta_{34}+\alpha_{23} \beta_{123}-\alpha_{23} \beta_{134}+\alpha_{124} \beta_{1}-\alpha_{124} \beta_{3}  
+\alpha_{12} \beta_{12}-\alpha_{12} \beta_{23}+\alpha_{14} \beta_{23}+\alpha_{14} \beta_{34}-\alpha_{234} \beta_{1} +\alpha_{234} \beta_{3}+\alpha_{23} \beta_{12}+\alpha_{23} \beta_{14} 
+\alpha_2 \beta_{123}-\alpha_2 \beta_{134}-\alpha_{34} \beta_{14}+\alpha_{34} \beta_{34}-\alpha_4 \beta_{123}+\alpha_4 \beta_{134}-2 \alpha_{124}+\alpha_{12} \beta_{1}+\alpha_{12} \beta_{3} 
-2 \alpha_{14} \beta_{3}+2 \alpha_{234}+2 \alpha_{23} \beta_{1}+2 \alpha_2 \beta_{12}+\alpha_2 \beta_{14}-\alpha_2 \beta_{23}-\alpha_{34} \beta_{1}-\alpha_{34} \beta_{3}+\alpha_4 \beta_{14}    -\alpha_4 \beta_{23}-2 \alpha_4 \beta_{34}+2 \beta_{123}-2 \beta_{134}-2 \alpha_{12}+3 \alpha_2 \beta_{1}+\alpha_2 \beta_{3}-2 \alpha_{34}+\alpha_4 \beta_{1}+3 \alpha_4 \beta_{3}+2 \beta_{14}  
+2 \beta_{23} -2 \alpha_2+2 \alpha_4+2 \beta_{1}-2 \beta_{3}+4,
\end{dmath*}
\begin{dmath*}
\omega_{21} =  \
\alpha_{123} \beta_{124}-\alpha_{123} \beta_{234}-\alpha_{134} \beta_{124}+\alpha_{134} \beta_{234}+\alpha_{123} \beta_{14}+\alpha_{123} \beta_{23}+\alpha_{12} \beta_{124}  
-\alpha_{12} \beta_{234}-\alpha_{134} \beta_{14}-\alpha_{134} \beta_{23}+\alpha_{34} \beta_{124}-\alpha_{34} \beta_{234}+\alpha_{123} \beta_{2}-\alpha_{123} \beta_{4} 
+\alpha_{12} \beta_{12}+\alpha_{12} \beta_{23}-\alpha_{134} \beta_{2}+\alpha_{134} \beta_{4}+\alpha_{14} \beta_{23}-\alpha_{14} \beta_{34}+\alpha_1 \beta_{124}-\alpha_1 \beta_{234} 
-\alpha_{23} \beta_{12}+\alpha_{23} \beta_{14}+\alpha_{34} \beta_{14}+\alpha_{34} \beta_{34}-\alpha_3 \beta_{124}+\alpha_3 \beta_{234}+2 \alpha_{123}+2 \alpha_{12} \beta_{2} -2 \alpha_{134} 
+\alpha_{14} \beta_{2}+\alpha_{14} \beta_{4}+\alpha_1 \beta_{12}+2 \alpha_1 \beta_{23}-\alpha_1 \beta_{34}-\alpha_{23} \beta_{2}-\alpha_{23} \beta_{4}-2 \alpha_{34} \beta_{4}+\alpha_3 \beta_{12} 
-2 \alpha_3 \beta_{14} -\alpha_3 \beta_{34}-2 \beta_{124}+2 \beta_{234}+2 \alpha_{14}+3 \alpha_1 \beta_{2}+\alpha_1 \beta_{4}+2 \alpha_{23}+\alpha_3 \beta_{2}+3 \alpha_3 \beta_{4}-2 \beta_{12} 
-2 \beta_{34}+2 \alpha_1 -2 \alpha_3-2 \beta_{2}+2 \beta_{4}+4,
\end{dmath*}
\begin{dmath*}
\omega_{22} =  \
\alpha_{123} \beta_{123}-\alpha_{123} \beta_{134}-\alpha_{134} \beta_{123}+\alpha_{134} \beta_{134}+\alpha_{123} \beta_{14}+\alpha_{123} \beta_{23}-\alpha_{134} \beta_{14} 
-\alpha_{134} \beta_{23}+\alpha_{14} \beta_{123}-\alpha_{14} \beta_{134}+\alpha_{23} \beta_{123}-\alpha_{23} \beta_{134}+\alpha_{123} \beta_{1}-\alpha_{123} \beta_{3} 
+\alpha_{12} \beta_{14}-\alpha_{12} \beta_{34}-\alpha_{134} \beta_{1}+\alpha_{134} \beta_{3}+\alpha_{14} \beta_{12}+\alpha_{14} \beta_{14}+\alpha_1 \beta_{123}-\alpha_1 \beta_{134} 
+\alpha_{23} \beta_{23}+\alpha_{23} \beta_{34}-\alpha_{34} \beta_{12}+\alpha_{34} \beta_{23}-\alpha_3 \beta_{123}+\alpha_3 \beta_{134}+2 \alpha_{123}+\alpha_{12} \beta_{1}+\alpha_{12} \beta_{3} 
-2 \alpha_{134}+2 \alpha_{14} \beta_{1}+\alpha_1 \beta_{12}+2 \alpha_1 \beta_{14}-\alpha_1 \beta_{34}-2 \alpha_{23} \beta_{3}-\alpha_{34} \beta_{1}-\alpha_{34} \beta_{3}+\alpha_3 \beta_{12} 
-2 \alpha_3 \beta_{23}-\alpha_3 \beta_{34}+2 \beta_{123}-2 \beta_{134}+2 \alpha_{12}+3 \alpha_1 \beta_{1}+\alpha_1 \beta_{3}+2 \alpha_{34}+\alpha_3 \beta_{1}+3 \alpha_3 \beta_{3}+2 \beta_{12} 
+2 \beta_{34}+2 \alpha_1-2 \alpha_3+2 \beta_{1}-2 \beta_{3}-4,
\end{dmath*}

\begin{dmath*}
\omega_{33} =  \
\alpha_{12} \beta_{12}-\alpha_{12} \beta_{14}-\alpha_{14} \beta_{12}+\alpha_{14} \beta_{23}+\alpha_{23} \beta_{14}-\alpha_{23} \beta_{34}-\alpha_{34} \beta_{23}+\alpha_{34} \beta_{34}  
+\alpha_{12} \beta_{2}+\alpha_{12} \beta_{4}-\alpha_{14} \beta_{1}-\alpha_{14} \beta_{3}-\alpha_1 \beta_{14}+\alpha_1 \beta_{23}+\alpha_{23} \beta_{1}+\alpha_{23} \beta_{3}+\alpha_2 \beta_{12} 
-\alpha_2 \beta_{34}-\alpha_{34} \beta_{2}-\alpha_{34} \beta_{4}-\alpha_3 \beta_{14}+\alpha_3 \beta_{23}+\alpha_4 \beta_{12}-\alpha_4 \beta_{34}-2 \alpha_{12}+2 \alpha_{14}-\alpha_1 \beta_{1} 
+\alpha_1 \beta_{2}-\alpha_1 \beta_{3}+\alpha_1 \beta_{4}+2 \alpha_{23}+\alpha_2 \beta_{1}+\alpha_2 \beta_{2}+\alpha_2 \beta_{3}+\alpha_2 \beta_{4}-2 \alpha_{34}-\alpha_3 \beta_{1} 
+\alpha_3 \beta_{2}-\alpha_3 \beta_{3}+\alpha_3 \beta_{4}+\alpha_4 \beta_{1}+\alpha_4 \beta_{2}+\alpha_4 \beta_{3}+\alpha_4 \beta_{4}-2 \beta_{12}+2 \beta_{14}+2 \beta_{23}-2 \beta_{34}+8.
\end{dmath*}
\medskip

It follows that
\smallskip

\begin{dmath*}
 f(\omega_{\alpha,\beta}) =   \ \omega_{11}+\omega_{12}+\omega_{21}-\omega_{22} - 2\omega_{33}  
=  -\alpha_{123} \beta_{123}+\alpha_{123} \beta_{124}+\alpha_{123} \beta_{134}-\alpha_{123} \beta_{234}+\alpha_{124} \beta_{123}-\alpha_{124} \beta_{124}-\alpha_{124} \beta_{134} 
+\alpha_{124} \beta_{234}+\alpha_{134} \beta_{123}-\alpha_{134} \beta_{124}-\alpha_{134} \beta_{134}+\alpha_{134} \beta_{234}-\alpha_{234} \beta_{123}+\alpha_{234} \beta_{124}  
+\alpha_{234} \beta_{134}-\alpha_{234} \beta_{234}-\alpha_{123} \beta_{1}+\alpha_{123} \beta_{2}+\alpha_{123} \beta_{3}-\alpha_{123} \beta_{4}+\alpha_{124} \beta_{1}  
-\alpha_{124} \beta_{2}-\alpha_{124} \beta_{3}+\alpha_{124} \beta_{4}+\alpha_{134} \beta_{1}-\alpha_{134} \beta_{2}-\alpha_{134} \beta_{3}+\alpha_{134} \beta_{4}-\alpha_1 \beta_{123}+\alpha_1 \beta_{124}  
+\alpha_1 \beta_{134}-\alpha_1 \beta_{234}-\alpha_{234} \beta_{1}+\alpha_{234} \beta_{2}+\alpha_{234} \beta_{3}-\alpha_{234} \beta_{4}+\alpha_2 \beta_{123}-\alpha_2 \beta_{124}-\alpha_2 \beta_{134}  
+\alpha_2 \beta_{234}+\alpha_3 \beta_{123}-\alpha_3 \beta_{124}-\alpha_3 \beta_{134}+\alpha_3 \beta_{234}-\alpha_4 \beta_{123}+\alpha_4 \beta_{124}+\alpha_4 \beta_{134}-\alpha_4 \beta_{234}  
-\alpha_1 \beta_{1}+\alpha_1 \beta_{2}+\alpha_1 \beta_{3}-\alpha_1 \beta_{4}+\alpha_2 \beta_{1}-\alpha_2 \beta_{2}-\alpha_2 \beta_{3}+\alpha_2 \beta_{4}+\alpha_3 \beta_{1}-\alpha_3 \beta_{2}  
-\alpha_3 \beta_{3}+\alpha_3 \beta_{4}-\alpha_4 \beta_{1}+\alpha_4 \beta_{2}+\alpha_4 \beta_{3}-\alpha_4 \beta_{4} ,
\end{dmath*}

\noindent and one can check that this coincides with the expansion of
\[ -\big(\alpha_{123} - \alpha_{124} - \alpha_{134} + \alpha_{234} + \alpha_1-\alpha_2-\alpha_3+\alpha_4
\big) \big(\beta_{123} - \beta_{124} - \beta_{134} + \beta_{234} + \beta_1-\beta_2-\beta_3+\beta_4
\big), \]
as needed.

\begin{remark}
It is instructive to follow our proof of entangleability on a concrete example. We do this on the simplest case relevant to quantum mechanics: a pair of qubits. This corresponds to two copies of the cone $\PSD_2$. It is useful to recall that under the canonical isomorphism $\Herm_m \otimes \Herm_n \simeq \Herm_{mn}$ we have the inclusions
\begin{equation} \label{equation:PSD}
\PSD_m \tmin \PSD_n \subset \PSD_{mn} \subset \PSD_m \tmax \PSD_n ,
\end{equation}
which are both strict for $m$, $n \geq 2$. The cone in the left-hand side of~\eqref{equation:PSD} is known as the cone of \emph{separable operators} and the cone in the right-hand side of~\eqref{equation:PSD} as the cone of \emph{block-positive operators} (see for example~\cite[Section 2.4.1]{ABMB}).

The cone $\PSD_2$ is isomorphic to the Lorentz cone $\CC(B_3)$, where $B_3$ is a $3$-dimensional Euclidean ball, commonly called the \emph{Bloch ball} in quantum information. A kite-square sandwiching is given by the maps $\Psi : \R^2 \times \R \to \Herm_2$, $\Phi : \Herm_2 \to \R^2 \times \R$, defined by
\[ \Psi(x,y\,;\,t) = \frac{1}{2} \begin{pmatrix} t+x & y \\ y & t-x \end{pmatrix}, \ \
 \Phi(A) = (A_{11}-A_{22}, A_{12}+A_{21} \,;\, A_{11}+A_{22}). \]
One checks that $\Psi(\CC(\kite_{\alpha})) \subset \PSD_2$ for $\alpha=(0,0,0,0)$, that $\Phi(\PSD_2) \subset \CC(\bsquare)$ and that $\Phi \circ \Psi = \Id$. The element $\omega$ in $\PSD_2 \tmax \PSD_2 \setminus \PSD_2 \tmin \PSD_2$ which is produced by the proof of Theorem~\ref{theorem:KS-entangleable} is constructed from the operators
\[ \rayS_1=\Psi(1,0\,;\,1)=\begin{pmatrix} 1 & 0 \\ 0 & 0 \end{pmatrix}, \
\rayS_2= \Psi(0,1\,;\,1)= \frac{1}{2} \begin{pmatrix} 1 & 1 \\ 1 & 1 \end{pmatrix}, \
\rayS_3= \Psi(-1,0\,;\,1)=\begin{pmatrix} 0 & 0 \\ 0 & 1 \end{pmatrix}, \
\]
via the formula
\[ \omega = \rayS_1 \otimes \rayS_2 - \rayS_2 \otimes \rayS_2 + \rayS_2 \otimes \rayS_1 + \rayS_3 \otimes \rayS_3 = 
\frac{1}{4} \begin{pmatrix} 3 & -1 & -1 & -1 \\ -1 & 1 & -1 & 1 \\ -1 & -1 & 1 & 1 \\ -1 & 1 & 1& 3 \end{pmatrix}.
\]
For this explicit example, a simple way to check that $\omega \not \in \PSD_2 \tmin \PSD_2$ is to show that $\omega \not \in \PSD_4$, which is the case since $\omega$ has an eigenvalue equal to $\frac{1-\sqrt{2}}{2}<0$. This is in fact expected, because $\omega$ has tensor rank $3$, and every such operator on $\Herm_2\otimes \Herm_2$ is separable whenever it is positive semidefinite~\cite[Theorem~3.2]{Cariello2015}.
\end{remark}

\section{Preparatory lemmata} \label{section:lemmas}

\subsection{A topological lemma}

Let $K \subset \R^n$ be a convex body, and $x$, $y$ in $K$. We say that $\{x,y\}$ is an \emph{antipodal pair} if there is a nonzero linear form $f$ on $\R^n$ such that
\begin{equation} \label{eq:antipodal} f(x) = \max_{K} f, \ f(y) = \min_{K} f .\end{equation}
We need a preliminary result on the existence of `sufficiently many' antipodal pairs. In what follows, we will denote with $[x,y]=\{ tx + (1-t)y \st t\in [0,1]\}$ the segment joining two points $x$ and $y$.

\begin{lemma} \label{lemma:antipodal}
Let $K \subset \R^n$ be a convex body. For every $z \in K$, there exists an antipodal pair $\{x, y\}$ such that $z \in [x,y]$. 
\end{lemma}

Note that for $n=2$ the topological argument in the following proof can be replaced by the intermediate value theorem.

\begin{proof}
We recall the following fact: if $z \in \inter(K)$, we can define the radial projection $R_z:K\setminus\{z\}\rightarrow \partial K$ by 
\[ \{R_z(x)\}=\{z+\lambda(x-z):\lambda\geq 0\}\cap \partial K.\]
That is, $R_z(x)$ is the intersection of $\partial K$ with the ray originating at $z$ and passing through $x$. Moreover, the function $R_z$ is continuous. To see this, define $\alpha_z(x)$ by the formula $R_z(x) = z+\alpha_z(x)^{-1}(x-z)$, and observe that the function $\alpha_z$ (which is the gauge functional of $K$ with respect to $z$) is convex, hence continuous.

We first assume that $K$ is \emph{regular} in the following sense: for every $\theta \in S^{n-1}$, there is a unique $F(\theta) \in \partial K$ which maximises $x \mapsto \scalar{x}{\theta}$ over $K$, and moreover $F : S^{n-1} \to \partial K$ is an onto homeomorphism. Note also that the minimum over $K$ of the function $x \mapsto \scalar{x}{\theta}$ is achieved at $F(-\theta)$, so that $\{F(\theta), F(-\theta)\}$ is an antipodal pair. Therefore, we need to prove that $K$ is equal to the set
\[ X \coloneqq \bigcup_{\theta \in S^{n-1}} [F(\theta),F(-\theta)]. \]

The surjectivity of $F$ implies that $\partial K \subset X$. Assume by contradiction that there is $z \in \mathrm{int}(K) \setminus X$. 
Define a map $H : [0,1] \times S^{n-1} \to S^{n-1}$ by the formula
\[  H(t,\theta) = F^{-1}\left(R_z\left((1-t)F(\theta)+t F(-\theta)  \right)\right).\]
It can be easily checked that $H$ is well defined (since $z \not \in X$) and continuous. Note that $H(0,\cdot)$ is the identity map on $S^{n-1}$, while $H(1/2,\cdot)$ is an even map on $S^{n-1}$ (i.e.\ $H(1/2,\theta)=H(1/2,-\theta)$). At this point we reach a contradiction, since the identity map cannot be homotopic to an even map: the identity has degree $1$, an even map has even degree, and the degree is a homotopy invariant~\cite[Section 2.2, especially Exercise 14]{Hatcher02}.

The extension to the general case relies on the following classical fact from convex geometry. For $\e>0$, we denote by $K_{(\e)}$ the $\e$-enlargement of a convex body $K$, i.e.\ the set of points at (Euclidean) distance at most $\e$ from $K$. 
\begin{lemma} \label{lemma:approximation}
If $K$ is a convex body in $\R^n$ and $\e>0$, there is a regular convex body $K'$ such that $K \subset K' \subset K_{(\e)}$.
\end{lemma}

Consider $K$ a general convex body. By Lemma~\ref{lemma:approximation}, there is a sequence $(K_k)_{k \geq 1}$ of regular convex bodies such that $K \subset K_k \subset K_{(1/k)}$ for every $k$. By the previous part, any $z \in K$ can be written as $z=t_k x_k + (1-t_k) y_k$ with $t_k \in [0,1]$ and $\{x_k, y_k\}$ an antipodal pair in $K_k$. This means that there exists $\theta_k \in S^{n-1}$ such that the functional $\scalar{\,\cdot\,}{\theta_k}$ is maximal on $K_k$ at $x_k$, and minimal at $y_k$. By compactness, up to extracting subsequences, we may assume that $t_k \to t$, $x_k \to x$, $y_k \to y$ and $\theta_k \to \theta$ as $k\to\infty$. We then have that $z= tx+(1-t)y$. Moreover, by uniform convergence the functional $\scalar{\,\cdot\,}{\theta}$ is maximal on $K$ at $x$, and minimal at $y$. It follows that $\{x,y\}$ is an antipodal pair in $K$. This proves the claim.
\end{proof}

Lemma~\ref{lemma:approximation} is a folklore result, which appears for example in~\cite{Klee59}. What we call regular is equivalent~\cite[Lemma 2.2.12]{Schneider13} to being both \emph{smooth} (i.e.\ such that every boundary point has a unique supporting hyperplane) and \emph{strictly convex} (i.e.\ such that the boundary does not contain a segment). When $0 \in  \inter(K)$, an approximation of $K$ by regular convex bodies is produced by the simple formula $\left( (K_{(\e)})^\circ_{(\e)} \right)^\circ$ as $\e \to 0$, where $\circ$ denotes the polarity in $\R^n$. For stronger approximation properties, see also~\cite[Theorem 3.4.1]{Schneider13}. 

\subsection{The parameter $\boldsymbol{\delta(K)}$}

We associate to each convex body a parameter which plays a central role in our proof of Theorem~\ref{theorem:KS}. We first recall standard definitions about the facial structure of convex bodies.

We denote by $\aff(X)$ the affine subspace spanned by a nonempty subset $X \subset \R^n$. If $A \subset \R^n$ is a closed convex subset, we denote by $\relint(A)$ and $\relbd(A)$ its \emph{relative interior} and \emph{relative boundary}, i.e.\ its interior and boundary when seen as a subset of $\aff(A)$. The following basic lemmata will we used multiple times.
\begin{lemma}[{\cite[Lemma~1.1.9]{Schneider13}}] \label{lemma:interior-convex}
Let $A \subset \R^n$ be convex. If $x \in \relint(A)$ and $y \in A$, then $\relint [x,y] \subset \relint(A)$.
\end{lemma}

\begin{lemma}[{\cite[Theorem 6.9]{Rockafellar70}}]\label{lemma:interior-pyramid}
Let $F \subset \R^n$ be a convex set and $x \in \R^n \setminus \aff(F)$. Then
\[ \relint \conv( F \cup \{x\}) = \{ \lambda x + (1-\lambda) y \st y \in \relint(F), \lambda \in (0,1) \}. \]
\end{lemma}

Fix a convex body $K \subset \R^n$. Let $F$ be a closed convex set with $F \subset K$. We say that $F$ is a \emph{face} if every segment contained in $K$ whose relative interior intersects $F$ is entirely contained in $F$. A face is \emph{proper} if $F \neq \emptyset$ and $F \neq K$. Every proper face is contained in $\partial K$. The dimension of $F$, denoted $\dim(F)$, is the dimension of $\aff(F)$.

An affine hyperplane $H \subset \R^n$ is a \emph{supporting hyperplane of $K$} if $H$ intersects $\partial K$ and is disjoint from $\mathrm{int}(K)$. A face is said to be \emph{exposed} if it is the intersection of $K$ with a supporting hyperplane. 

A \emph{maximal face} is a face which is maximal (with respect to set inclusion) among proper faces. Every proper face is contained in a maximal face, and every maximal face is exposed.

For $0 \leq d \leq n-1$, we say that a $x \in \partial K$ is $d$-\emph{extreme} (resp.\ $d$-\emph{exposed}) if it is contained in a face of dimension at most $d$ (resp.\ in an exposed face of dimension at most $d$). Note that any boundary point is $(n-1)$-exposed and therefore $(n-1)$-extreme. By an extreme (resp.\ exposed) point we mean a $0$-extreme (resp.\ $0$-exposed) point, i.e.\ a point $x \in \partial K$ such that $\{x\}$ is a face (resp.\ an exposed face). We denote by $\Ext(K)$ the set of extreme points of $K$.

\begin{definition}
For $K \subset \R^n$ a convex body, denote by $\delta(K)$ the smallest $d$ such that there exist an extreme point $u \in K$ and a $d$-extreme point $v \in K$ satisfying $[u,v] \cap \inter(K) \neq \emptyset$ (by Lemma~\ref{lemma:interior-convex}, this is equivalent to saying that $\relint [u,v] \subset \inter(K)$).
\end{definition}

A theorem by Asplund~\cite[Theorem 2.1.7]{Schneider13} states that any $d$-extreme point is the limit of a sequence of $d$-exposed points. It follows that $\delta(K)$ can be equivalently defined as the smallest $d$ such that there exist an exposed point $u \in K$ and a $d$-exposed point $v \in K$ with the property that $[u,v]$ intersects $\mathrm{int}(K)$.

It is easy to check that $\delta(K) \leq n-1$ for every convex body $K \subset \R^n$. A \emph{simplex} in $\R^n$ is a convex body with $n+1$ extreme points (a convex body $K$ is a simplex if and only if the cone $\CC(K)$ is classical). If $K$ is a simplex, we have $\delta(K)=n-1$ (this is because if $F$ is a face of a simplex $K$ with $\dim(F) \leq n-2$ and $x \in \Ext(K) \setminus F$, then $\conv(F \cup \{x\})$ is a proper face, and therefore does not intersect the interior). We show that this property characterises simplices.

\begin{proposition} \label{proposition:deltaK}
Let $K \subset \R^n$ be a convex body that is not a simplex. Then $\delta(K) \leq n-2$.
\end{proposition}

The following lemma, which appears in~\cite[Proposition 8]{Plavala16}, will be used in the proof of Proposition~\ref{proposition:deltaK}. We include here a proof for convenience.

\begin{lemma} \label{lemma:maximal-face}
Let $K \subset \R^n$ be a convex body that is not a simplex. Then there is a maximal face $F \subset K$
such that $\card(\Ext(K) \setminus F) \geq 2$.
\end{lemma}
\begin{proof}
Let $A \subset \Ext(K)$ be a set of $n+1$ affinely independent extreme points. Since $K$ is not a simplex, there exists $x \in \Ext(K) \setminus A$. Choose a maximal face $F \subset K$ such that $x \not \in F$ (for example choose for $F$ a maximal face containing $y$, where $y \in \partial K$ is such that $[x,y] \cap \mathrm{int}(K) \neq \emptyset)$. Suppose by contradiction that $\card(\Ext(K) \setminus F)=1$, which means that $\Ext(K) \setminus F = \{x\}$. It follows that $A \subset F$, and therefore $\R^n = \aff (A) \subset \aff(F)$, a contradiction.
\end{proof}

In the next proof we will repeatedly use the following fact: let $F\subseteq \partial K$ be closed and convex (e.g.\ let it be a face), and pick $z\in \relint(F)$. If $H$ is a supporting hyperplane containing $z$, then $F\subset H$.

\begin{proof}[Proof of Proposition~\ref{proposition:deltaK}]
Let us first assume that there exists a maximal face $F \subset K$ with $\dim F\leq n-2$. Take $x \in \Ext(K) \setminus F$ and $y \in \relint (F)$ (if $\dim F= 0$ we have $F=\{y\}$). We claim that $[x,y]$ intersects $\mathrm{int}(K)$. Suppose by contradiction that $[x,y] \subset \partial K$. Consider $z = (x+y)/2 \in \relint [x,y]\subset \partial K$, and let $H$ be a supporting hyperplane containing $z$. Necessarily both $x$ and $y$ belong to $H$, and therefore $K \cap H$ is a face -- in fact, an exposed face -- containing $F \cup \{x\}$. Since $K\cap H\neq K$, we contradict the maximality of $F$.

\begin{figure}[htbp] \begin{center}
\begin{tikzpicture}[scale=1]
	\coordinate (a) at (-2,0);
	\coordinate (b) at (0,-1);
	\coordinate (c) at (2.2,-0.2);
	\coordinate (d) at (2,0.5) ;
	\coordinate (e) at (-1,0.8);
	\coordinate (x) at (-1,4);
	\coordinate (y) at (2,3);
	\coordinate (z) at (0.6,0.9);
	\coordinate (w) at (1,1.5);
	\coordinate (u) at (2.1,0.125);
	\draw[dotted, fill=gray!30] (a)--(b)--(c)--(d)--(e)--(a);
	\draw (a)--(x);
	\draw (b)--(x);
	\draw (c)--(x);
	\draw (d)--(x);
	\draw[dotted] (e)--(x);
	\draw (a)--(b)--(c)--(d);
	\draw (x) node {$\bullet$};
	\draw (y) node {$\bullet$};
	\draw (-1,0) node {$F$};
	\draw (z) node {$\bullet$};
	\draw (z) node [left] {$y$};
	\draw (x) node [above] {$x_1$};
	\draw (y) node [above] {$x_2$};
	\draw (w) node [right] {\, $y'$};
	\draw[dashed] (z)--(w);
	\draw (w)--(y);
	\draw (w) node {$\bullet$};
	\draw[very thick] (c)--(d);
	\draw[densely dashed] (x)--(u);
	\draw (u) node {$\bullet$};
	\draw (u) node [right] {$z$};
	\draw (2.3,0.5) node {$F'$};
	\draw (-1.8,2) node {$C_1$};
	\end{tikzpicture}
\end{center} 
\caption{Proof of Proposition~\ref{proposition:deltaK} when all maximal faces have dimension $n-1$}
\label{figure-deltaK}
\end{figure}
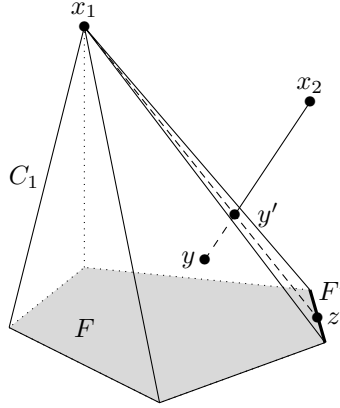

Let us now assume that all maximal faces have dimension $n-1$. Let $F$ be a maximal face given by Lemma~\ref{lemma:maximal-face}, take $x_1 \neq x_2$ in $\Ext(K) \setminus F$ (see Figure~\ref{figure-deltaK}).
Define $C_1 = \conv(F \cup \{x_1\})$ and $C_2 = \conv(F \cup \{x_2\})$, which are convex bodies in $\R^n$. These convex bodies have a common interior point $y$. (Otherwise, by applying the Hahn--Banach separation theorem to the nonempty open convex sets $\inter(C_1)$ and $\inter(C_2)$, one could find a linear form $\ell$ and a real number $t$ such that $\inter(C_1) \subset \{\ell <t\}$ and $\inter(C_2) \subset \{\ell >t\}$. Since a convex body is the closure of its interior, the hyperplane $\{ \ell =t\}$ would contain $C_1 \cap C_2 \supset F$ and therefore equal $\aff(F)$. This is a contradiction because $\aff(F)$ does not separate $\{x_1\}$ and $\{x_2\}$). Since $y \in \inter(C_1)$ and $x_2 \not \in C_1$, the segment $[y,x_2]$ intersects $\partial C_1$ at a point $y'\in \relint [y,x_2]$. By Lemma~\ref{lemma:interior-convex}, $y' \in \inter(C_2)$. We may write $y' = \lambda x_1 + (1-\lambda) z$ for $\lambda \in (0,1)$ and $z \in F$. Since $y' \in \partial C_1$, Lemma \ref{lemma:interior-pyramid} implies that $z \in \relbd(F)$. Consequently, there is a face $F' \subsetneq F$ such that $z \in F'$. On the other hand, $y' \in \inter ( C_2) \subset \inter(K)$, and therefore we have that $\delta(K) \leq \dim (F') \leq n-2$, as needed.
\end{proof}

\subsection{Convex geometry}

We use a couple of elementary lemmata which we state and prove now.

\begin{lemma} \label{lemma:boundary-inclusion}
Let $L_1$, $L_2$ be convex bodies in $\R^n$ such that $\partial L_1 \subset \partial L_2$. Then $L_1 = L_2$. 
\end{lemma}

\begin{proof}
Since $L=\conv(\partial L)$ for every convex body $L$, we immediately deduce from our hypothesis that $L_1 \subset L_2$.
Suppose by contradiction that the inclusion is strict, and pick $x \in L_2 \setminus L_1$. Choose $y \in \inter(L_1)$, and note that since $x \not \in L_1$, there is $z \in \relint [x,y] \cap \partial L_1$. Then, our assumption implies that $z \in \partial L_2$. On the other hand, since $x \in L_2$ and $y \in \inter(L_1) \subset \inter(L_2)$, we have that $\relint[x,y] \subset \inter(L_2)$ and therefore that $z \in \inter(L_2)$, a contradiction.

Alternatively, a purely topological proof goes as follows: both $\partial L_1$ and $\partial L_2$ are homeomorphic to $S^{n-1}$; since $S^{n-1}$ is not homeomorphic to any of its proper subsets, we conclude that $\partial L_1 = \partial L_2$. 
\end{proof}

\begin{lemma} \label{lemma:boundary-section}
Let $K \subset \R^n$ be a convex body, and let $E \subset \R^n$ be an affine subspace which intersects $\inter(K)$. Then $\relint (K \cap E) = \inter(K) \cap E$ and $\relbd (K \cap E) = \partial K \cap E$.
\end{lemma}

\begin{proof} (See also \cite[Corollary 6.5.1]{Rockafellar70}.)
The inclusion $\inter(K) \cap E \subset \relint(K \cap E)$ is simple and holds in full generality. Assume now that $E$ intersects $\inter(K)$, and let $y \in \inter(K) \cap E$. For every $x \in \relint(K \cap E)$, there exists $\e>0$ such that we have $x + \e(x-y) \in K \cap E$. Since the relative interior of the segment $[x +\e (x-y),y]$ is contained in $\inter(K)$, we obtain that $x \in \inter(K)$. This proves the first assertion, and the second follows by taking complements inside $K \cap E$.
\end{proof}

\subsection{Projective transformations}

We will rely on very basic properties of projective transformations in $\R^n$, which we now introduce in an elementary way, referring to~\cite[\S 2.6]{Ziegler95} for more details. We think of projective transformations as the effect on a convex body $K$ of a linear bijective transformation acting on $\CC(K)$. A \emph{projective transformation} is a map $P$ of the form
\[ x \mapsto \frac{B(x)+z}{\langle w, x\rangle + k}, \] where $B:\R^n\rightarrow \R^n$ is a linear map, $z,w\in \R^n$ and $k\in \R$, and moreover 
\[ \det\left(
\begin{array}{cc}
B&z\\
 w^t&k
\end{array}
\right)\neq 0.\]
This map is defined on $\R^n\setminus H$, where $H=\{x\in \R^n:\langle w, x\rangle + k=0$\}, and extends to an automorphism of the projective space. We say that $P$ is well defined on a convex body $K \subset \R^n$ if $K\cap H=\emptyset$.

A projective transformation $P$ preserves properties such as exposedness or extremality of points. Moreover, the cones $\CC(K)$ and $\CC(P(K))$ are isomorphic. 

\begin{lemma} \label{lemma:send-to-infinity}
Let $K$ be a convex body in $\R^n$, and $H_1$, $H_2$ be supporting hyperplanes of $K$, such that $K \cap H_1 \cap H_2 = \emptyset$. Then there is a projective transformation $P$, which is well defined on $K$, such that the supporting hyperplanes $P(H_1)$ and $P(H_2)$ of $P(K)$ are parallel (in the language of projective geometry, $P(H_1)$ and $P(H_2)$ intersect at infinity). \end{lemma}

\begin{proof}
Simply send $H_1 \cap H_2$ to infinity. In more detail: suppose the hyperplanes are given by $H_i = \{x\in \R^n: f_i(x) = t_i\}$ for linear functionals $f_1$, $f_2$ and real numbers $t_1$, $t_2$, with the property that $f_i(x) \geq t_i$ for every $x\in K$. For an arbitrary $x_0 \in K$, a suitable choice is the projective transformation
\[ P : x \mapsto \frac{x-x_0}{f_1(x)+f_2(x)-t_1-t_2} \]
which is well defined on $K$ since $K\cap H_1 \cap H_2 = \emptyset$. According to the notation above, we must check that 
\[ \det\left(
\begin{array}{cc}
\Id&-x_0\\
 (f_1+f_2)^t&-t_1-t_2
\end{array}
\right)\neq 0.\]
This determinant can be easily computed and equals $-t_1-t_2+f_1(x_0)+f_2(x_0)$, which is positive since $x_0\in K$.

Once we have checked that the projective matrix $P$ is well defined, since $H_1\cap H_2\subset H=\{x\in \R^n: (f_1+f_2)(x)-t_1-t_2=0\}$, we conclude that $P(H_1)\cap P(H_2)=P(H_1 \cap H_2)$ contains only points at infinity. The fact that $P(H_1)$ and $P(H_2)$ are supporting hyperplanes of $P(K)$ follows trivially from the properties of projective maps.
\end{proof}

\section{Proof of Theorem~\ref{theorem:KS}} \label{section:main-proof}

We first prove the easy part of Theorem~\ref{theorem:KS}: a classical cone $\C$ does not admit a kite-square sandwiching. For this we use the fact that $\C$ enjoys the \emph{decomposition property}: whenever the equation $x_1+x_2=y_1+y_2$ is satisfied for $x_1$, $x_2$, $y_1$, $y_2 \in \C$, there exist $z_{11}$, $z_{12}$, $z_{21}$, $z_{22} \in \C$ such that $x_i=z_{i1}+z_{i2}$ and $y_j=z_{1j}+z_{2j}$, for $i,j=1$ or $2$. (This property actually characterises classical cones, see e.g.~\cite[Theorem~2.1]{Effros67}). Assume by contradiction that there is a kite $\kite_\alpha$ and maps $\Phi : V \to \R^3$, $\Psi : \R^3 \to V$ such that $\Phi \circ \Psi = \Id$, $\Psi(\CC(\kite_\alpha)) \subset \C$ and $\Phi(\C) \subset \CC(\bsquare) \subset \CC([-1,1]^2)$.
We consider the vectors $\rayS_1$, $\rayS_2$, $\rayS_3$, $\rayS_4$ introduced in the proof of Proposition~\ref{prop:g-kite}. These vectors generate the extreme rays of the cone $\CC(\kite_\alpha)$, and have the extra property that $\rayS_1+\rayS_3=\rayS_2+\rayS_4$. Note that $\Psi(\rayS_i) \in \C$. Since $\Psi(\rayS_1)+\Psi(\rayS_3)=\Psi(\rayS_2)+\Psi(\rayS_4)$, the decomposition property implies the existence of $z_{12}$, $z_{14}$, $z_{32}$, $z_{34}$ in $\C$ such that
\[ \Psi(\rayS_1) = z_{12}+z_{14},\ \Psi(\rayS_2) = z_{12}+z_{32},\ \Psi(\rayS_3) = z_{32}+z_{34},\ \Psi(\rayS_4) = z_{14}+z_{34}. \]
It follows that
 \[ \rayS_1 = \Phi(z_{12})+\Phi(z_{14}),\ \rayS_2 = \Phi(z_{12})+\Phi(z_{32}),\ \rayS_3 = \Phi(z_{32})+\Phi(z_{34}),\ \rayS_4 = \Phi(z_{14})+\Phi(z_{34}). \]
Each vector $\Phi(z_{ij})$ belongs to $\CC(\bsquare)$ and therefore to $\CC([-1,1]^2)$. We label the $4$ facets of $\CC([-1,1]^2)$ as
\begin{gather*}
F_1 = \{ (t,u \,;\, t) \in \R^3 \st |u| \leq t \}, \quad F_2 = \{ (u,t \,;\, t) \in \R^3 \st |u| \leq t \} \\
F_3 = \{ (-t,u \,;\, t) \in \R^3 \st |u| \leq t \}, \quad F_4 = \{ (u,-t \,;\, t) \in \R^3 \st |u| \leq t \}. 
\end{gather*}
We have that $\rayS_i \in F_i$ for $1 \leq i \leq 4$. It follows from the definition of a face that $\Phi(z_{ij}) \in F_i \cap F_j$ for every $i$, $j$. Since $F_i \cap F_j \cap \CC(\bsquare)=\{0\}$, we obtain that $\Phi(z_{ij})=0$, implying that $\rayS_i=0$, a contradiction.

\medskip

\begin{remark} \label{remark:why-blunt}
Using the blunt square instead of the full square $S=[-1,1]^2$ when defining a kite-square sandwiching is critical to the validity of the previous claim. Indeed, consider the matrix
\[ M = \frac{1}{2} \begin{bmatrix} 1 & 1 & 1 \\ 1 & -1 & 1 \\ -1 & 1 & 1 \\ -1 & -1 & 1\end{bmatrix} . \]
Let $\Psi : \R^3 \to \R^4$ the linear map associated to $M$ and $\Phi : \R^4 \to \R^3$ the linear map associated to the transpose $M^T$. We can check that $\Phi \circ \Psi = \Id$, $\Psi(\CC(\kite_\alpha)) \subset \R_+^4$ for $\alpha=(0,0,0,0)$, and that $\Phi(\R_+^4) = \CC(S)$. Hence, we see that the classical cone $\R_+^4$ does admit a factorisation analogous to the kite-square sandwiching if we replace the blunt square by the full square. However, we have seen in the previous paragraph that it does not admit a kite-square sandwiching in the standard sense.
\end{remark}

\medskip

We move on to the proof of the remaining implication in Theorem~\ref{theorem:KS}. We argue that if $K$ is a convex body which is not a simplex, then $\CC(K)$ admits a kite-square sandwiching. This statement is equivalent to the `only if' part of Theorem~\ref{theorem:KS}. Let $K \subset \R^n$ be a convex body which is not a simplex, and set $d=\delta(K)$. By Proposition~\ref{proposition:deltaK}, we know that $0 \leq d \leq n-2$. By the definition of $\delta(K)$ and the remark following it, there exist an exposed point $v_1 \in K$ and a $d$-exposed point $v' \in K$ such that $[v_1,v'] \cap \inter(K) \neq \emptyset$. Let $H_1$ and $H_2$ be exposing hyperplanes, i.e.\ such that $H_1 \cap K = \{v_1\}$ and $F \coloneqq H_2 \cap K$ is a $d$-dimensional face containing $v'$.

By applying a projective transformation, we may assume that $H_1$ and $H_2$ are parallel (see Lemma~\ref{lemma:send-to-infinity} for details; note that $\delta(P(K)) = \delta(K)$ whenever $P$ is a projective transformation that is well defined on $K$, and that the existence of a kite-square sandwiching for $\CC(K)$ and $\CC(P(K))$ are equivalent since these cones are isomorphic). By further applying an affine transformation, we may therefore reduce to the situation where $v'=0$, $H_1=f^{-1}(1)$ and $H_2=f^{-1}(0)$, for a linear form $f$ satisfying $0 \leq f \leq 1$ on $K$.

Let $V_1$ be the $1$-dimensional linear space spanned by $v_1$ and $V_2$ be the $d$-dimensional linear space spanned by $F$. Note that $V_1 \cap V_2 = \{0\}$, because $f(v_1)\neq 0$ and $V_2\subseteq H_2 = f^{-1}(0)$. Let $V = V_1 \oplus V_2 = \mathspan(V_1 \cup V_2)$. We also note that $H_1 \cap K = \{v_1\}$, $H_2 \cap K = V_2 \cap K = F$, and $0 \in F$.

\begin{claim} \label{claim:pyramid}
We have that $K \cap V = \conv(F \cup \{v_1\})$.
\end{claim}

\begin{proof}
If $d=0$, which means that $F=\{0\}$, it is very easy to see that $K \cap V=[0,v_1]=\conv(0, v_1)$, from which the claim follows.
Let us then assume that $d\geq 1$. Then, we can conclude thanks to Lemma~\ref{lemma:boundary-inclusion}, which we apply with $L_1 = \conv(F \cup \{v_1\})$ and $L_2 = K \cap V$, seen as convex bodies in $V$. To see that both are convex bodies in $V$, it suffices to observe that they are convex compact sets whose affine hull equals $V$, because $V = \aff(F \cap \{v_1\}) \subset \aff(L_1) \subset \aff(L_2) \subset V$. 
We now explain why the hypothesis $\relbd(L_1) \subset \relbd(L_2)$ also holds, which allows us to apply Lemma~\ref{lemma:boundary-inclusion}. Since $V \cap \inter(K) \neq \emptyset$, we have $\relbd(K \cap V) =\partial K \cap V$ by Lemma~\ref{lemma:boundary-section}. Therefore, it remains to justify that 
\begin{equation}  \relbd (\conv(F \cup \{v_1\})) \subset \partial K. \label{eq:relative-boundaries} \end{equation} 
By Lemma \ref{lemma:interior-pyramid}, we have that
\begin{equation*}  \relbd ( \conv(F \cup \{v_1\}) )= F \cup \{v_1\} \cup \Big\{ \lambda v_1 + (1-\lambda)x \st x \in \relbd(F), \lambda \in (0,1) \Big\}.
\end{equation*}
It is obvious that $F \subset \partial K$ and $v_1 \in \partial K$. Choose now $\lambda \in (0,1)$ and $x \in\nobreak \relbd(F)$, and let $G$ a proper face of $F$ containing $x$. Then $G$ is also a face of $K$, and since $\dim(G) < \dim(F) = \delta(K)$, it follows from the minimality in the definition of $\delta(K)$ that $[v_1,x] \cap \inter(K) = \emptyset$, or equivalently $[v_1,x] \subset \partial K$. This proves~\eqref{eq:relative-boundaries} and completes the proof of the claim.
\end{proof}

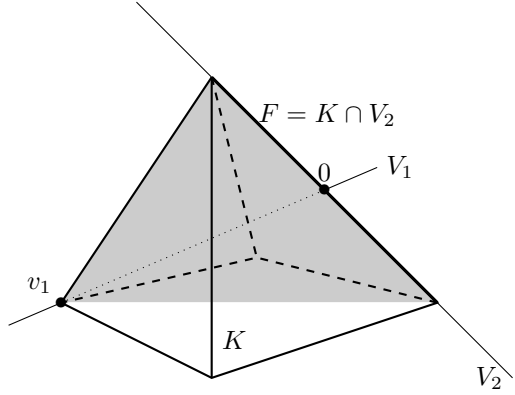
\begin{figure}[htbp] \begin{center}
\begin{tikzpicture}[scale=1]
	\coordinate (a) at (0,3);
	\coordinate (b) at (-2,0);
	\coordinate (c) at (0,-1);
	\coordinate (d) at (3,0) ;
	\coordinate (e) at (0.6,0.6);
	\coordinate (o) at (1.5,1.5);
	\draw[white,fill=gray!40] (b)--(d)--(a)--(b) ;
	\draw[thick] (a)--(b)--(c)--(d)--(a) ;
	\draw[thick] (a)--(c) ;
	\draw[thick,dashed] (d)--(e)--(b) ;
	\draw[thick,dashed] (a)--(e) ;
	\draw[very thick] (a)--(d) ;
	\draw (-1,4)--(4,-1) ;
	\draw (-2.7,-0.3)--(b) ;
	\draw (o)--(2.2,1.8) ;
	\draw[dotted] (b)--(o) ;
	\draw (o) node {$\bullet$};
	\draw (o) node[above] {$0$};
	\draw (b) node {$\bullet$};
	\draw (b) node[above left] {$v_1$};
	\draw (4,-1) node[left] {$V_2$};
	\draw (2.2,1.8) node[right] {$V_1$};
	\draw (.5,2.5) node[right] {$F=K \cap V_2$};
	\draw (0.3,-0.5) node {$K$};
	\end{tikzpicture}
\end{center}
\caption{Illustration for the proof when $K \subset \R^3$ is a pyramid over a square. We have $\delta(K)=1$. The section $K \cap V$ is depicted in gray. In that case $\dim(W)=1$ and $L$ is a segment with $0$ in the interior. }
\end{figure}

Choose an arbitrary subspace $W \subset \R^n$ such that $\R^n = V \oplus W$. Note that $\dim (W) = n-(d+1) \geq 1$. Let $\pi$ be the projection with range $W$ and kernel $V$. Denote $L \coloneqq \pi(K)$, so that $L$ is a convex body in $W$. Indeed, (a) $L$ is clearly convex and compact; and (b) picking $z\in [0,v_1] \cap \inter(K) \neq \emptyset$, since $0,v_1\in \ker (\pi)$ we have that $0=\pi(z)\in \pi(\inter(K)) \subseteq \relint (\pi(K)) = \relint (L)$.

By Lemma~\ref{lemma:antipodal}, there is an antipodal pair $\{x_1,x_2\}$ in $L$ such that $0 \in [x_1,x_2]$. More precisely, there is a linear form $\ell$ on $W$ such that
\[ \ell(x_1) =  \min_{L} \ell <0 < \max_{L} \ell =\ell(x_2) \]
(the inequalities are strict since $\dim (W) \geq 1$ and $0 \in \relint(L)$). Without loss of generality (replace $\ell$ by a suitable positive multiple), we can assume that $\ell(x_2)-\ell(x_1)=\nobreak 1$. Call $\mu = \ell(x_2) \in (0,1)$, so that $\ell(x_1)=\mu-1$. Since $0 \in [x_1,x_2]$, by looking at the action of $\ell$ one sees that necessarily $\mu x_1 +(1-\mu)x_2 = 0$

Consider preimages $y_1$, $y_2$ in $K$ such that $\pi(y_1)=x_1$ and $\pi(y_2)=x_2$. We have that $\mu y_1 + (1-\mu)y_2 \in K \cap \ker(\pi) = K \cap V$. By Claim~\ref{claim:pyramid}, there is $0 \leq \lambda \leq 1$ and $v_2 \in F$ such that 
\begin{equation} \label{eq:y-v-relation} \mu y_1 + (1-\mu) y_2 = \lambda v_1 + (1-\lambda)v_2. \end{equation}
Applying $f$ to the previous equation yields $\mu f(y_1) + (1-\mu) f(y_2) = \lambda$. 
Since $x_i = \pi(y_i)\neq 0$ (otherwise e.g.\ $\ell(x_i)=0$) neither of $y_1,y_2$ belongs to $K\cap V$, which implies that $f(y_i)\in (0,1)$ (because $f=1$ on $K$ only at $v_1$ and $f=0$ only on $F$) and hence that $\lambda \in (0,1)$.

We are going to produce a kite-square sandwiching for $\CC(K)$ out of this situation. Define a linear map $\Psi : \R^n \times \R \to \R^2 \times \R$ by the formula
\[ \Psi(x\,;\,t) = \Big(t-2f(x),(1-2\mu)t+2\ell(\pi(x))\,;\,t\Big).\]
We claim that $\Psi(\CC(K)) \subset \CC(\bsquare)$. It is enough to check that $\Psi(x\,;\,1) \subset \bsquare \times \{1\}$ for every $x \in K$, i.e.\ that
\begin{equation} \label{eq:points-in-blunt-square} \Big(1-2f(x), 1-2\mu+2\ell(\pi(x))\Big) \end{equation}
belongs to the blunt square $\bsquare$. On the set $K$, the functional $f$ takes on values in $[0,1]$ and $\ell \circ \pi$ takes on values in $[\mu-1,\mu]$, so each coordinate in~\eqref{eq:points-in-blunt-square} belongs to $[-1,1]$. It remains to check that they cannot be $\pm 1$ simultaneously. Indeed, if the first coordinate equals $\pm 1$, i.e.\ if $x \in K$ is such that $f(x) \in \{0,1\}$, then $x \in \{v_1\} \cup F \subset V$, so that $\ell(\pi(x))=0$; together with the fact that $\mu\in (0,1)$, this shows that the second coordinate is in $(-1,1)$. Therefore, $\Psi(\CC(K)) \subset \CC(\bsquare)$.

We will now construct a map $\Phi : \R^2 \times \R \to \R^n \times \R$ such that $\Psi \circ \Phi= \Id$. This is straightforward, since we only need to pick a suitable kite $\kite_\alpha \subset \bsquare$ and map it to $K$. Define numbers $(\alpha_i)_{1 \leq i \leq 4}$ in $(-1,1)$ by the formulae
\begin{gather*}
(1,\alpha_1\,;\,1) = (1,1-2\mu\,;\,1) = \Psi(v_2\,;\,1), \\
(\alpha_2,1\,;\,1) = (1-2f(y_2),1\,;\,1) = \Psi(y_2\,;\,1), \\
(-1,\alpha_3\,;\,1) = (-1,1-2\mu\,;\,1) = \Psi(v_1\,;\,1), \\
(\alpha_4,-1\,;\,1) = (1-2f(y_1),-1\,;\,1) = \Psi(y_1\,;\,1) ,
\end{gather*}
and consider the kite $\kite_{\alpha}$. We note that
\begin{equation} \label{eq:alpha-relation} \mu(\alpha_4,-1\,;\,1) + (1-\mu)(\alpha_2,1\,;\,1) 
= \lambda (-1,\alpha_3\,;\,1) + (1-\lambda) ( 1,\alpha_1\,;\,1).
\end{equation}
We define a linear map $\Phi : \R^2 \times \R \to \R^n \times \R$ by requiring that
\[ \Phi(1,\alpha_1\,;\,1)=(v_2\,;\,1),\ \Phi(\alpha_2,1\,;\,1)=(y_2\,;\,1),\ \Phi(-1,\alpha_3\,;\,1)=(v_1\,;\,1),\ \Phi(\alpha_4,-1\,;\,1)=(y_1\,;\,1).\]
One checks that $\Phi$ is well defined by comparing equations~\eqref{eq:alpha-relation} and~\eqref{eq:y-v-relation}, and by observing that $\Psi \circ \Phi= \Id$.
It is clear that $\Phi(\CC(\kite_\alpha)) \subset \CC(K)$, since by definition of $\Phi$ this is satisfied for each of the $4$ extreme rays of $\CC(\kite_{\alpha})$. We have checked all the conditions for the existence of a kite-square sandwiching, and the proof of Theorem~\ref{theorem:KS} is therefore complete.

\section*{Acknowledgements}

We are very grateful to Kyung Hoon Han for several remarks which helped us clarifying the manuscript. We thank also Alexander M\"uller-Hermes for useful comments on some of the results.
GA was supported in part by ANR (France) under the grant StoQ (2014-CE25-0003). LL acknowledges financial support from the European Research Council under the Starting Grant GQCOP (Grant no.~637352), from the Foundational Questions Institute under the grant FQXi-RFP-IPW-1907, and from the Alexander von Humboldt Foundation. CP is partially supported by Spanish MINECO through Grant No.~MTM2017-88385-P, by the Comunidad de Madrid through grant QUITEMAD-CM P2018/TCS4342 and by SEV-2015-0554-16-3. MP acknowledges support from grant VEGA 2/0142/20, from the grant of the Slovak Research and Development Agency under contract APVV-16-0073, from the Deutsche Forschungsgemeinschaft (DFG, German Research Foundation - 447948357) and the ERC (Consolidator Grant 683107/TempoQ).

\bibliography{entangleability}{}

\begin{thebibliography}{10}
\bibitem{Plavala16}
M.~Pl\'{a}vala.
\newblock All measurements in a probabilistic theory are compatible if and only
  if the state space is a simplex.
\newblock {\em Phys. Rev. A}, 94(4):042108, 7, 2016.
\end{thebibliography}
\bibliographystyle{plain}

\end{document}